\documentclass[12pt]{amsart}

\usepackage{graphicx}              % to include figures
\usepackage{amsmath}               % great math stuff
\usepackage{amsfonts}              % for blackboard bold, etc
\usepackage{amsthm}                % better theorem environments
\usepackage{mathrsfs}
\usepackage{tikz-cd}
\usepackage{hyperref}
\usepackage[all,cmtip]{xy}
\entrymodifiers={+!!<0pt,\fontdimen22\textfont2>}
\usetikzlibrary{matrix,arrows,decorations.pathmorphing}
\usepackage[italian,english]{babel}
\usepackage[utf8]{inputenc}
\usepackage{MnSymbol}
\usepackage[autostyle]{csquotes}
\newtheorem{thm}{Theorem}[section]
\newtheorem{lem}[thm]{Lemma}
\newtheorem{prop}[thm]{Proposition}
\newtheorem{cor}[thm]{Corollary}

\newtheorem{rmk}[thm]{Remark}
\newtheorem{qst}[thm]{Question}
\newtheorem*{phil}{Philosophy}

\newtheorem*{ack}{Acknowledgments}

\theoremstyle{definition}

\newtheorem{hyp}[thm]{Hypothesis}

\newcommand{\PP}{\mathbb{P}}
 
\newcommand{\OO}{\mathcal{O}}

\newcommand{\MM}{\overline{\mathcal{M}}}

\newcommand{\CH}{\mathrm{CH}}
\newcommand{\U}{\mathcal{U}}
\newcommand{\F}{\mathcal{F}}
\newcommand{\w}{\omega}
\newcommand{\ch}{\mathrm{Ch}}
\newcommand{\se}{\sigma}
\newcommand{\im}{\mathrm{Im}}
\newcommand{\Z}{\mathcal{Z}}

\numberwithin{equation}{thm}

\begin{document}
\date{}
\nocite{*}

\title{\textbf{1-Cycles on Fano varieties}}

\author {cristian minoccheri}
\address{Department of Mathematics, Stony Brook University, Stony Brook, NY 11794}
\email{cristian.minoccheri@stonybrook.edu}

\author {xuanyu pan}
\address{Institute of Mathematics, AMSS, Chinese Academy of Sciences, 55 ZhongGuanCun East Road, Beijing, 100190, China and Max Planck Institute for Mathematics, Vivatsgasse 7, Bonn, Germany 53111}
\email{pan@amss.ac.cn}
\email{panxuanyu@mpim-bonn.mpg.de}
\date{\today}

\begin{abstract}
We prove results about 1-cycles on certain Fano varieties using techniques that rely on rational curves. Firstly, we show that Fano weighted complete intersections with index bigger than half their dimension have trivial first Griffiths group. Secondly, we prove that the first Chow group of most $2$-Fano weighted complete intersections, and of $2$-Fano conic-connected varieties in $\PP^n$ of high enough index (with $3$ obvious exceptions), are generated by lines. Furthermore, if the Fano variety of lines is irreducible, the first Chow group is isomorphic to $\mathbb{Z}$. \end{abstract}

\maketitle
\tableofcontents

\section{Introduction}

\textbf{First Griffiths group.}
Understanding the geometry of higher codimensional cycles on algebraic varieties is both a natural goal and a fairly difficult one to achieve. Even considering 1-cycles, there are still many open questions. In this note, we answer some of these questions for 1-cycles on some Fano varieties. We will always work over the field of complex numbers.

In section $2$, we provide an answer to a question of Voisin about the first Griffiths group in many new cases. Recall that, for a variety $X$, we have an inclusion of abelian groups $\CH_1(X)_{alg} \subset \CH_1(X)_{hom}$, and one defines the first Griffiths group Griff$_1(X)$ as the quotient of $\CH_1(X)_{hom}$ by $\CH_1(X)_{alg}$. Griffiths groups have attracted interest -- among other things -- as a way to tackle rationality problems. Voisin raised the following:
\begin{qst} \label{Voisin}
Do all Fano varieties have trivial first Griffiths group?

\end{qst}
In general, to answer this question seems quite difficult. However, there are no known counterexamples, and there are affirmative answers in dimension at most 3 (\cite{BS}, which requires a deep result in algebraic K-theory), and for Fano complete intersections in projective space (\cite{TZ}). In this paper, building on previous results of \cite{Min}, we answer the question for many Fano complete intersections in weighted projective space:

\begin{thm}
\label{thmA}
Let $X:=X_{d_1,...,d_c} \subset \PP_{\mathbb{C}}(1,1,1,a_3,...,a_n)$ be a smooth weighted complete intersection of dimension at least $3$ which is not isomorphic to a (standard) projective space. Let $\iota_X$ denote the index of $X$. If $$\iota_X > \frac{1}{2}dim(X),$$ then \upshape{Griff}$_1(X)=0$.

\end{thm}

An explicit example is given by cyclic covers of projective space of low degree (precisely, in the lower half of the Fano bound).

The fundamental idea is fairly simple: we use bend and break to write a rational curve as algebraically equivalent to a chain of curves of degree $1$; then, we use the fact from \cite{Min} that any two curves of degree $1$ meeting at a point are algebraically equivalent (since the fibers of the evaluation map are connected).

\

\textbf{First Chow group.}
In section $4$, we study the first Chow group of $2$-Fano varieties (i.e., such that $c_1(X)$ and $ch_2(X)$ are positive). The main technique will involve studying a Kontsevich moduli space of conics and the fibers of its total evaluation map. In particular, we will make use of a useful formula for the canonical bundle of the fiber of such evaluation map, for which we provide a direct proof in section $3$. Our main application is the following:

\begin{thm}
Let $X$ be a smooth, $2$-Fano, non-linear variety. Assume that either: 

1) $X_{d_1,...,d_c} \subset \PP(1,1,1,a_3,...,a_n)$ is a weighted complete intersection, with $d_1+...+d_c \leq a_3+...+a_n$, or

2) $X \subset \PP^n$ is conic-connected, with index $\iota_X \geq \frac{n+4}{2}$, or

3) $X \subset \PP^n$ has Picard number $1$ and index $\iota_X \geq \frac{3dim(X)+5}{4}$.

Then $\CH_1(X)$ is generated by lines. If furthermore the Fano variety of lines is irreducible, $\CH_1(X) \simeq \mathbb{Z}$. In particular, if $X$ is a general cyclic cover of projective space, $\CH_1(X) \simeq \mathbb{Z}$.
\end{thm}

The method we used involves constructing a fibration over the conic-connected variety (or weighted complete intersection) with generic fiber a space of broken conics $\Delta$, and showing that $\Delta$ is Fano to produce a section (by \cite{GHS}). The extra hypotheses concerning the index guarantee that $\Delta$ is connected (whereas for weighted complete intersections it follows from results in \cite{Min}).

\

While rational Chow groups of $1$-cycles on Fano varieties are better understood, the integral analogue of those results is more subtle. 

Let us first take a look, more generally, at the behavior of higher codimensional cycles. Already for $0$-cycles, there are some interesting phenomena. Koll\'ar, Miyaoka and Mori proved that, if $X$ is Fano, $\CH_0(X) \simeq \mathbb{Z}$ (\cite{KMM}). On the other hand, the geometry of zero cycles on surfaces which are not Fano can be already fairly complicated. For instance, Mumford proved that if $S$ is a K3 surface, then $\CH_0(S)$ has infinite dimension (\cite{mumford}). With respect to $1$-cycles, it is proved in \cite{TZ} that the first Chow group of a rationally connected variety is generated by rational curves, and that the first Chow group of $2$-Fano (i.e., such that $\ch_2(X)>0$) complete intersections in projective space is isomorphic to $\mathbb{Z}$. With respect to $2$-cycles, we know from \cite{Pan} that the second Chow group of hypersurfaces within the $3$-Fano range (i.e., such that $\ch_3(X)>0$) is isomorphic to $\mathbb{Z}$ (assuming furthermore that the Fano variety of lines is smooth). In the special case of cubic hypersurfaces, there is recent work of Mboro (\cite{Rene}).

These observations and Voisin's Question \ref{Voisin} suggest the following philosophy:
\begin{phil}
 The more positive the tangent bundle of a variety is, the simpler the Chow groups are.
 \end{phil}

Varieties that satisfy strong positivity conditions of the tangent bundle are higher Fano varieties (introduced by de Jong and Starr in \cite{HF}): one says that a variety $X$ is $m$-Fano if $\ch_k(X)>0$ for all $k=1,...,m$. By \cite{HF} and \cite{AC}, we know that $2$-Fano manifolds satisfy \enquote{higher} versions of similar properties of Fano manifolds; for instance, they admit a rational surface through every point. Based on the philosophy above, we are lead to formulate the following:
\begin{qst}
For $X$ a $2$-Fano variety with Picard number $1$ and pseudo-index at least $3$, is $\CH_1(X) \simeq \mathbb{Z}$?

\end{qst}

In this paper, we study this question for \em broken-conic-connected \em 2-Fano manifolds whose space of broken conics through $2$ general points is connected, and we prove that their first Chow group is generated by curves of degree $1$ with respect to some fixed ample line bundle. Our main application is to the cases considered in Theorem 1.3.

\begin{ack}
\textup{The first author would like to thank Jason Starr for many useful conversations, constant encouragement, and helpful suggestions, especially with respect to the proof of Corollary 4.7. The second author is thankful to Prof.~Pandharipande and Prof.~Voisin for the invitation to the conference Cycles and Moduli which is held in ETH. Some thoughts of this paper are dicussed with some participants in this conference. The second author is also grateful to his truly great friend Zhiyu Tian and Prof.~Starr for their encouragement and help through these years. Some parts of this paper are written up when the second author is in the Max Planck Institute for Mathematics and Institute of Mathematics of AMSS. The second author is grateful to both institutes for the support.}
\end{ack}

\section{First Griffiths group of weighted complete intersections}

Let us first fix the notation. We will denote by  $\PP(a_0,...,a_n)$ the (complex) weighted projective space, with weights $a_0,...,a_n$, and we will work with complete intersections $X_{d_1,...,d_c}$ in such space of some fixed degrees $d_1,...,d_c$. For a review of the basics on weighted projective spaces, we refer to \cite{Min}, or the classical sources \cite{Dol} and \cite{Mor}. Throughout this section we will always assume the following:

\begin{hyp}[Basic hypothesis]
Let $X:=X_{d_1,...,d_c} \subset \PP(a_0,...,a_n)$ be a smooth weighted complete intersection of degrees $d_1,...,d_c$. We will assume that the following conditions are satisfied:

(1) $X$ has dimension at least $3$, i.e. $c \leq n-3$;

(2) $a_0=a_1=a_2=1$;

(3) $a_3+...+a_n +2+c-n \leq d_1+...+d_c$;  and

(4) $d_1+...+d_c \leq a_3+...+a_n$.
\end{hyp}

Note that inequality (3) is not restrictive: if $a_3+...+a_n +2+c-n > d_1+...+d_c$, then $X$ would be isomorphic to a projective space by \cite{CMSB}. Inequality (2) should also not be restrictive: from all the examples we have (see \cite{PS}), for a weighted complete intersection to be smooth, at least three of the weights have to be $1$ (and often, many more).

Inequality (4) tells us that the degrees have to be small; in particular, $X_{d_1,...,d_c}$ is Fano. With respect to Question 1.1, (4) should be replaced by the exact Fano bound: $d_1+...+d_c \leq 2+ a_3+...+a_n$. On the other hand, our method requires $X_{d_1,...,d_c}$ to be covered by lines, therefore the bound $d_1+...+d_c \leq 1+ a_3+...+a_n$ would be necessary anyway.

\

The main ingredient to prove triviality of the first Griffiths groups of $X$ will be the connectedness of the fibers of the $1$-pointed evaluation map for certain Kontsevich spaces of stable maps to $X$. More precisely, let $\alpha$ denote a curve class of degree 1 on $X$. Then there are coarse moduli spaces $\MM_{0,r}(X, m \alpha)$ that parametrize $r-$pointed stable maps to $X$ with curve class $m \alpha$ (see for example \cite{FP} or \cite{CK} for definitions and basic properties). One also has evaluation maps $ev_r^m: \MM_{0,r}(X,m \alpha) \to X^r$ that associate to a map the image in $X^r$ of its $r$ marked points. 

\

The following result was essentially proven in \cite{Min}: 

\begin{prop}
\label{basic}
Let $X_{d_1,...,d_c} \subset \PP(a_0,...,a_n)$ be a smooth weighted complete intersection that satisfies the \em basic hypothesis\em. Let $\alpha$ be a curve class of degree $1$. Then every geometric fiber of $$ev_1^1: \MM_{0,1}(X,\alpha) \to X$$ is non-empty and connected.
\end{prop}

We now state the main Theorem of this section:

\begin{thm}
\label{thmA}
Let $X_{d_1,...,d_c} \subset \PP(a_0,...,a_n)$ be a smooth weighted complete intersection that satisfies the \em basic hypothesis\em. Assume further that $$d_1+...+d_c \leq a_3+...+a_n + \frac{4-\dim(X_{d_1,...,d_c})}{2},$$ or, equivalently, that $$\iota_{X_{d_1,...,d_c}}>\frac{1}{2}dim(X_{d_1,...,d_c}).$$ Then \upshape{Griff}$_1(X_{d_1,...,d_c})=0$.
\end{thm}

The extra condition on the index is required to apply bend-and-break. 

As an explicit example, consider smooth, degree $r$ covers of $\PP^{n-1}$ branched along a hypersurface of degree $r \cdot a$ (with $a,b \geq 2$), i.e., smooth weighted hypersurfaces $X_{ra} \subset \PP(1,...1,a)$. Then, by Theorem \ref{thmA}, \upshape{Griff}$_1(X_{ra})=0$ for $2a(r-1) < n$, i.e., for cyclic covers within exactly half of the Fano bound $a(r-1)<n$.

\begin{rmk}
\normalfont
In \cite{TZ}, Tian and Zong prove that the first Griffiths group of complete intersections in (standard) projective space is trivial. A direct application of their the method to the weighted case could be possible (if one managed to apply a connectedness result of Hartshorne to weighted complete intersections), and one could ask what results one would get. Our method wouldn't work as well as the one by Tian and Zong in the case of complete intersections in projective space. For example, the bound we get in the case of a hypersurface is $d \leq \frac{n+1}{2},$ whereas the optimal bound is $d \leq n-1$. On the other hand, a direct application of their method in the weighted case would require bounds that involve the greatest common divisor of the weights, and therefore are not easily made explicit. In many instances, those bounds are worse than what we get here; for example, they are not satisfied if many weights have a large common factor.
\end{rmk}

In the proof of Theorem  \ref{thmA}, we will need the following fact about the topology of smooth weighted complete intersections (see \cite[Proposition 6]{Dim1}, \cite[Definition 1 and Corollary 10]{Dim2}):

\begin{prop}
\label{homology}
With the same notation as above, $H_2(X, \mathbb{Z}) \simeq \mathbb{Z} \alpha$.
\end{prop}

We can finally proceed to the proof of Theorem \ref{thmA}:

\begin{proof}
By \cite[Theorem 1.3]{TZ}, $\mathrm{CH}_1(X)$ is generated by rational curves. Let $C$ be a rational curve; by Proposition \ref{homology}, $C$ is homologous to $N\alpha$ for some integer $N$.

On the other hand, assume that $C$ has class $m\alpha$, with $m \geq 2$, and consider a morphism $f: \PP^1 \to X$ such that $f_*([\PP^1])=m \alpha$. If $f(0)=p$ and $f(\infty)=q$, the expected dimension of $$\mathrm{Mor}(\PP^1,X; 0 \mapsto p, \infty \mapsto q)$$ at $[f]$ is $$\chi(\PP^1,T_X(-2))= m (\sum a_i - \sum d_j) - (n-c) \geq 2 (\sum a_i - \sum d_j) - (n-c) \geq 2,$$ so that we can use bend-and-break. Therefore $C$ is algebraically equivalent to a connected union of curves of degrees less than $m$. By iterating this process, $C$ is algebraically equivalent to a connected sum of curves with class $\alpha$. Finally, since every fiber of $ev_1^1$ is connected by Proposition 2.2, two curves with class $\alpha$ that meet at one point are algebraically equivalent. Therefore, $$\mathrm{CH}_1(X)_{hom}=\mathrm{CH}_1(X)_{alg}.$$
\end{proof}

\

\section{Canonical Bundle} 

In this section, $X$ will be a smooth projective variety of dimension $n$ with an ample line bundle $\OO_X(1)$. We will assume that $X$ contains a line $\alpha$ with respect to this an ample line bundle, i.e., $\alpha\cdot \OO_X(1)=1$. Let us fix the notation as follows.

Let $\F$ be a (non-empty) general fiber of $ev_2^2: \MM_{0,2}(X,2\alpha) \to X \times X$ over $p$ and $q$. Let $\pi: \U\rightarrow \F$ be the universal family with two universal sections $\se_0,\se_1: \F \rightarrow \U$ and the universal map $f:\U \rightarrow X$. Therefore, we have $\im (f\circ \se_0)=\{p\}$ and $\im (f\circ \se_1)=\{q\}$. Let $\w_{\U/\F}$ be the relative dualizing sheaf of $\U$ over $\F$, which we denote by $\w$. Denote the singular locus of $\pi$ by $\Z$. The locus $\Z$ is defined by the first fitting ideal of the relative differential sheaf $\Omega^1_{\U/\F}$, see \cite[Definition 5.1 and Lemma 5.2]{Pan1}.

\begin{hyp} \label{hypo} In the following of this section, we assume that
\begin{enumerate}
\item The restriction of $f$ to $\pi^{-1}(b)$ is an embedding conic (possibly singular) with respect to $\OO_X(1)$ for every $b\in \F$.
\item The moduli stack $\MM_{0,2}(X,2\alpha)$ is automorphism-free, smooth and of expected dimension $\chi(\pi^{-1}(b),N_{\pi^{-1}(b)/X})+2$ at every $b\in \F$, where $N_{\pi^{-1}(b)/X}$ is the normal bundle of $\pi^{-1}(b)$ in $X$. 

\item The boundary divisor $\Delta\subseteq \F$ which parametrizes broken-conics is smooth.
\end{enumerate}
\end{hyp}
Therefore, the fiber $\F$ is smooth and of pure dimension \begin{equation}\label{dimen}
\chi(\pi^{-1}(b),N_{\pi^{-1}(b)/X}) -2(n-1).
\end{equation} It follows from \cite[Proposition 5.4]{Pan1} that $\U$ is a smooth projective variety. Since the restriction of $f$ to $\pi^{-1}(b)$ is an embedding conic for every $b\in \F$, we have a short exact sequence
\begin{equation} \label{conormal}
0\rightarrow N^{\vee}_{\U/X} \rightarrow f^*\Omega_X \rightarrow \Omega _{\U/\F} \rightarrow 0
\end{equation} 
where $\Omega$ is the sheaf of differentials and $N^{\vee}_{\U/X} $ is the relative (locally free) conormal bundle. Again, as the stable map $f|_{\pi^{-1}(b)}$ (whose image is denoted by $\U_b$) is an embedding for every $b\in \F$, the first order deformation space of $f$ is given by $H^0(\U_b, N_{\U_b/X})$ (cf. \cite[Page 61]{GHS}) where $N_{\U_b/X}$ is the normal bundle. Therefore, the first order deformation space of $f$ fixing $p$ and $q$ is given by $H^0(\U_b, N_{\U_b/X}(-p-q))$ and the relative version gives the tangent space of $\F$, i.e., $T_{\F}= \pi_* (N_{\U/X}(-\se_0-\se_1 ))$.
\begin{lem} \label{cyrel} With the notations as above, we have that
\begin{enumerate}
\item $td_{\leq 2}(T_{\pi})= 1-\frac{1}{2}c_1(\w)+\frac{1}{12}\left((c_1(\w))^2+[\Z]\right)$, where $td_{\leq 2}$ denotes the sum of the Todd classes up to degree $2$.
\item $\ch_{\leq 2}(N_{\U/X})= (n-1)-f^*K_X +c_1(\w) +f^* \ch_2(\Omega_X)-\frac{1}{2} c_1(w)^2+[\Z]$, where $\ch_{\leq2}$ denotes the sum of the Chern characters up to degree $2$.
\item \[\begin{aligned}
\ch_{\leq 2}(N_{\U/X}(-\se_0-\se_1)) & =n-1 -(n-1)(\se_0+\se_1)-f^*K_X+c_1(\w)\\
&+(f^*K_X-c_1(\w))(\se_0+\se_1)+ \frac{n-1}{2}(\se_0+\se_1)^2\\
&+f^*\ch_2(\Omega_X)-\frac{1}{2}c_1(\w)^2+[\Z].
\end{aligned}\]

\end{enumerate}
\end{lem}
\begin{proof}
The assertion (1) follows from \cite[Lemma 5.6 and Lemma 5.7]{Pan1}. 
For (2), we have
\[
\begin{aligned}
\ch_{\leq 2} & =\mathrm{rank}(N_{\U/X})-\ch_1(N_{\U/X}^{\vee})+\ch_2(N_{\U/X}^{\vee} )\\
&=n-1-c_1(f^*\Omega_X-\Omega_{\U/\F})+\ch_2(f^*\Omega_X)-\ch_2(\Omega_{\U/\F})\\
&=n-1-f^*K_X+c_1(\w)+\ch_2(f^*\Omega_X)-\frac{1}{2}c_1(w)^2+[\Z]
\end{aligned}
\]
where the second identity follows from the exact sequence (\ref{conormal}) and the last identity follows from \cite[Lemma 5.7]{Pan1}.

Note that $\ch_{\leq 2}(\OO_{\U}(-\se_0-\se_1))=1-\se_0-\se_1+\frac{1}{2}(\se_0+\se_1)^2$.
The assertion (3) follows from (2) and \[\ch(N_{\U/X}(-\se_0-\se_1))=\ch(N_{\U/X})\cdot \ch(\OO_{\U}(-\se_0-\se_1)).\]
\end{proof}

\begin{lem} \label{dual}
With the same notation as above, we have that
\[c_1(\w)=-\se_0-\se_1.\]
\end{lem}
\begin{proof}
It suffices to show that $\w(\se_0+\se_1)$ is trivial. In fact, for any point $b\in \F$, we have that the dualizing sheaf $\w_{\U_b}$ of $\U_b$ is $\OO_{\U_b}(-2)$. Therefore, the restriction $\w(\se_0+\se_1)|_{\U_b}=\w_{\U_b}(p+q)$ is trivial. This implies that there exists a line bundle  $L$ on $\F$ such that $\w(\se_0+\se_1)=\pi^*L$. Therefore, the line bundle $L=\se_0^*(\w(\se_0+\se_1))$ is equal to
\[
\begin{aligned}
\se_0^*\w \otimes \se_0^* \OO_{\U}(\se_0) &=\se_0^*(\Omega_{\U/\F})\otimes \se_0^*(N_{\se_0/\U}) \\
 &= \se_0^*(\Omega_{\U/\F})\otimes \se_0^*(T_{\U/\F}) \\
 &=\OO_{\F} 
\end{aligned}
\]
where $N_{\se_0/\U}$ is the normal bundle of $\se_0$ in $\U$. Here we are using the fact that $\se_0$ is disjoint with $\se_1$, and that $N_{\se_0/\U}=(T_{\F/\U})|_{\se_0}$.
\end{proof}

\begin{lem} \label{van}
With the same notation as above, we have $$R^i\pi_*(N_{\U/X}(-\se_0-\se_1 ))=0$$ for $i\geq 1.$
\end{lem}
\begin{proof}
For $i\geq 2$, it is obvious. To show $R^1\pi_*(N_{\U/X}(-\se_0-\se_1 ))=0$, it suffices to show that $H^1(\U_b, N_{\U_b/X}(-p-q))=0$ for every $b\in \F$. By the Riemann-Roch theorem for vector bundles on cuves and (\ref{dimen}), we have 
\[\chi(\U_b,N_{\U_b/X}(-p-q)) =\chi(\pi^{-1}(b),N_{\pi^{-1}(b)/X}) -2(n-1)=\dim(\F).\]
Note that $\F$ is smooth and the tangent space to $\F$ at the point $b$ is given by the first order deformation space $H^0(\U_b, N_{\U_b/X}(-p-q))$. Therefore, we have that $$\dim(\F)=\dim H^0(\U_b, N_{\U_b/X}(-p-q)).$$ It follows that $H^1(\U_b, N_{\U_b/X}(-p-q))=0$.
\end{proof}

\begin{prop}\label{canonical} With the same notation as above, we have
$$-K_{\F}=\pi_*f^*(\ch_2(\Omega_X))+\Delta$$
in $\CH(\F)^1_{\mathbb{Q}}$.
\end{prop}
\begin{proof}
Since $\U$ and $\F$ are smooth projective varieties, we can apply the Grothendieck-Riemann-Roch theorem to the morphism $\pi : \U\rightarrow \F$. In general, we have
\[\pi_*\left(\ch(\mathcal{G})\cdot td(T_{\pi})\right)=\ch(\pi_{!}(\mathcal{G}))\] where $\mathcal{G}$ is a coherent sheaf on $\U$ and $\pi_{!}\mathcal{G}=\sum \limits_{i=0}^1 (-1)^i R^i\pi_*(\mathcal{G}).$ It follows from Lemma \ref{van} and Lemma \ref{cyrel} (1), (3) that
\[\begin{aligned}
\ch_1(T_{\F})&=\ch_1\left(\pi_*(N_{\U/X}(-\se_0-\se_1))\right) \\
&=\pi_*\left[ \left( \ch(N_{\U/X}(-\se_0-\se_1))\cdot td(T_{\pi})\right)_2 \right]
&=\pi_*(A+B+C),
\end{aligned}\]
where 
\[
\begin{aligned}
A&=(f^*K_X-c_1(\w))(\se_0+\se_1)+ \frac{n-1}{2}(\se_0+\se_1)^2\\
&+f^*\ch_2(\Omega_X)-\frac{1}{2}c_1(\w)^2+[\Z],\\
B&=\frac{n-1}{12}(c_1(\w)^2+[\Z]),\\
C&=\frac{n-1}{2}(\se_0+\se_1)c_1(\w)+\frac{1}{2}f^*K_X\cdot c_1(\w)-\frac{1}{2}c_1(\w)^2.
\end{aligned}
\]
Since $\pi_*([\Z])=\Delta$ by \cite[Proposition 5.3]{Pan1}, and $\Delta+\pi_*\left(c_1(\w)^2\right)=0$ by \cite[Lemma 6.4]{Pan1}, we have that $\pi_*B=0$.
Using Lemma \ref{dual}, we get
 \[A+C=\frac{1}{2}f^* K_X\cdot (\se_0+\se_1)+f^*\ch_2(\Omega_X^1)+[\Z].\]
Note that, since $\im(f\circ \se_i)$ is a point, $f^*K_X\cdot \se_i=0$ for $i=0,1$. 
It follows that
\[-K_{\F}=\pi_*f^*(\ch_2(\Omega_X))+\Delta.\]
\end{proof}

\begin{rmk}
\normalfont
Our method could give a general formula for $\ch_k(\F)$ for any $k$, similar to the formula in \cite[Proposition 1.1]{AC}, but the calculation would be tedious. For the purpose of this paper, the formula for $k=1$ is sufficient.
\end{rmk}

\section{Fano Families of conics and 1-cycles}

Throughout this section, we will assume that $X$ is a smooth, $2$-Fano, non-linear, projective variety with an ample line bundle $\OO_X(1)$. Let $\alpha$ be a curve class of $\OO_X(1)$-degree $1$. Out of convenience, we will call \enquote{lines} and \enquote{conics} curves with classes $\alpha$ and $2\alpha$ respectively.

\

Let $B$ be the full boundary divisor of $\MM_{0,2}(X, 2\alpha)$, which has multiple irreducible components. As customary, we can denote the components of $B$ by $D(a,b,i \alpha,j \alpha)$, where $a$, $b$ are the number of points on the first and second component respectively of the domain of the stable map, and $i \alpha$, $j \alpha$ are the curve classes of the stable map restricted to the first and second component respectively. Then $B$ is the union of $D(1,1,\alpha,\alpha)$, $D(2,0,\alpha,\alpha)$ and $D(2,0,0,2\alpha)$.

\

Now let $\F$ be a general fiber of $ev_2^2: \MM_{0,2}(X, 2\alpha) \to X \times X$ over a general point $(p, q) \in X \times X$. Since $p$ and $q$ are distinct, $\F$ cannot intersect $D(2,0,0,2\alpha)$. Since $p$ and $q$ are general (and $X$ is not linear), there is no line through $p$ and $q$, which means that $\F$ cannot intersect $D(2,0,\alpha,\alpha)$. Therefore $\F$ intersects only one irreducible component of the full boundary $B$, and this component is $D(1,1,\alpha,\alpha)$, which we will denote from now on by $\MM_{0,2}(X,\PP^1\vee \PP^1)$.

\

We will assume furthermore that $X$ is \em broken-conic-connected\em, i.e., that the restriction of the natural evaluation map $ev_2^2: \MM_{0,2}(X,2\alpha) \to X \times X$ to $\MM_{0,2}(X,\PP^1\vee \PP^1)$ is dominant.

\

A central role will be played by $2-$pointed conics on $X$, which are rather special. Namely, the class $2\alpha$ is \enquote{$2-$minimal}, meaning that there are no curves on $X$ through $2$ general points of lower degree. One can define more generally $n-$minimal curves, which have been studied by de Jong and Starr, and which enjoy good properties (see \cite{dJS06a}, section $5$). 

\begin{prop}\label{checkhyp}
Let $X$ be a smooth, $2$-Fano, non-linear, projective variety which is broken-conic-connected. Consider the evaluation map $ev_2^2: \MM_{0,2}(X,2\alpha) \to X \times X$ and let $\F$ its fiber over a general point $(p,q) \in X \times X$. Then:

(1) $\F$ is smooth, it parametrizes automorphism free stable maps every component of which is free, and $\MM_{0,2}(X,2\alpha)$ is smooth of the expected dimension at every point of $\F$;

(2) the intersection of $\F$ with the boundary divisor is a smooth divisor.
\end{prop}
\begin{proof}
(1) Since $p$ and $q$ are general points of $X$ and in particular distinct, a point $[C,f,p_1,p_2]$ of $\F$ can only parametrize either a stable map of degree $1$ from $\PP^1$ to a conic, or a stable map of degree $1$ from a tree of two copies of $\PP^1$ with one marked point on each component to two lines meeting at one point. In either case, since every component of $C$ has degree $1$ over its image, $[C,f,p_1,p_2]$ is automorphism free. Also, since in either case the image of every component of $C$ passes through a general point if $X$, every component of $[C,f,p_1,p_2]$ is free.

\

By \cite[Theorem 2]{FP}, this implies that $\MM_{0,2}(X,2\alpha)$ is smooth of the expected dimension at every point of $\F$. Furthermore, if $\mathcal{U} \subset \MM_{0,2}(X,2\alpha)$ is the open subset parametrizing unions of free curves, by Generic Smoothness we have that $ev_2^2|_{\mathcal{U}}$ is smooth. On the other hand, $\F$ is contained in the fiber of $ev_2^2|_{\mathcal{U}}$ over $(p,q)$; thus $\F$ is smooth.

\

(2) Let $\epsilon: \MM_{0,2}(X,\PP^1\vee \PP^1) \to X \times X$ be the restriction of $ev_2^2$ to $\MM_{0,2}(X,\PP^1\vee \PP^1)$, which is an irreducible component of the boundary divisor of $\MM_{0,2}(X,2\alpha)$, and the only component of the boundary divisor that $\F$ intersects. Since $X$ is broken-conic-connected, $\epsilon$ is dominant. Let $\mathcal{V} \subset \MM_{0,2}(X,\PP^1\vee \PP^1)$ be the open subset parametrizing unions of free curves. By Generic Smoothness, the restriction of $\epsilon$ to $\mathcal{V}$ is smooth. On the other hand, every point in $\F \cap \MM_{0,2}(X,\PP^1\vee \PP^1)$ parametrizes a union of free curves; therefore, $\F \cap \MM_{0,2}(X,\PP^1\vee \PP^1) \subset \mathcal{V}$, which implies that $\F \cap \MM_{0,2}(X,\PP^1\vee \PP^1)$ is smooth.
\end{proof}

We will denote the smooth divisor $\F \cap \MM_{0,2}(X,\PP^1\vee \PP^1)$ by $\Delta$.

\begin{lem}
\label{fanolemma}
Let $X$ be a smooth, $2$-Fano, non-linear, projective variety which is broken-conic-connected. Assume that the divisor $\Delta$ is connected. Then $\Delta$ is Fano.
\end{lem}
\begin{proof}
Recall that we have morphisms:
\begin{center}
\begin{tikzcd}\mathcal{U} \arrow{d}{\pi}  \arrow{r}{f} 
&X \\
\F
\end{tikzcd}
\end{center}
where $\mathcal{U}$ is the universal family over the general fiber $\F$ of $ev_2^2$.

By Proposition \ref{checkhyp}, Hypothesis \ref{hypo} is satisfied. It follows from Proposition \ref{canonical} that $$-K_\F=\pi_*f^*(\ch_2(X))+\Delta.$$ By adjunction, $$-K_{\Delta}=(-K_{\F}-\Delta)|_{\Delta}= (\pi_*(f^* \ch_2(X)))|_{\Delta}.$$ By \cite[Lemma 2.7]{AC}, $\pi_*f^*$ preserves ampleness; therefore, since $X$ is 2-Fano, we have that $\Delta$ is Fano.
\end{proof}

We will now construct a Fano fibration $\phi:Y \to X$ (by which we mean a morphism whose general fiber is Fano), which we will use to compute the first Chow group of $X$.

\begin{prop}
Let $X$ be a smooth, $2$-Fano, non-linear projective variety with an ample line bundle $\OO_X(1)$. Let $\alpha$ be a curve class on $X$ with $\OO_X(1)$-degree $1$. Assume that $X$ is broken-conic-connected, and that $\Delta$ is connected. Then $X$ admits a Fano fibration $\phi: Y \to X$, where $Y$ is a family of broken conics.
\end{prop}
\begin{proof}
Let \[h :\MM_{0,2}(X,\PP^1\vee \PP^1)\rightarrow \MM_{0,1}(X,\PP^1\vee \PP^1)\] be the forgetful map that maps a reducible conic with one marked point on each component to a reducible conic with one marked point on the first component.

We have a commutative diagram of evaluation maps:
\[\xymatrix{\MM_{0,2}(X,\PP^1\vee \PP^1) \ar[d]^{h} \ar[rr]^{(ev_1,ev_2)} && X\times X\ar[d]^{\pi_1}\\
\MM_{0,1}(X,\PP^1\vee \PP^1)\ar@/^1pc/[u]^{\sigma_1} \ar[rr]^{ev} && X}.\]
where $\sigma_1$ is the universal section of $h$ and $\pi_1$ is the projection onto the first factor.

Let $M_{1,p}$ be the fiber  $ev^{-1}(p)$ of $ev$ over a general point $p\in X$. Then we have maps:
\[
\xymatrix{ \MM_{0,2}(X,\PP^1\vee \PP^1)|_{M_{1,p}}\ar[rr]^{ev_2} \ar[d] &&X\\
M_{1,p} \ar@/^1pc/[u]^{\sigma_1}}
\]

\

Now we define $Y$ as $\MM_{0,2}(X,\PP^1\vee \PP^1)|_{M_{1,p}}$ and $\phi$ as $ev_2$. It is now easy to show that the morphism $\phi:Y \to X$ is a Fano fibration.

Namely, it suffices to note that, for $p$ and $q$ general points in $X$, the fiber $\phi^{-1}(q)$ of $\phi$ over $q$ is equal to the fiber $(ev_1,ev_2)^{-1}(p,q)$. This fiber is in turn equal to $\Delta$, which is Fano by Lemma \ref{fanolemma}.
\end{proof}

We can now prove the main result of this section:

\begin{thm}
Let $X$ be a smooth, $2$-Fano, non-linear projective variety with an ample line bundle $\OO_X(1)$. Let $\alpha$ be a curve class on $X$ with $\OO_X(1)$-degree $1$. Assume that $X$ is broken-conic-connected, and that $\Delta$ is connected. Then $\CH_1(X)$ is generated by lines. In particular, if furthermore the Fano scheme of $X$ is rationally chain connected, $\CH_1(X) \simeq \mathbb{Z}$.
\end{thm}
\begin{proof}
Let $C$ be an irreducible curve in $X$. We will show that $C$ is rationally equivalent to a linear combination of lines. We can assume that $C$ passes through a general point of $X$ (see \cite{TZ}, proof of Proposition 3.1). Namely, one can find a complete intersection of very ample divisors which equals $C \cup C'$, where $C'$ is an irreducible curve that contains a general point; then $[C]$ is rationally equivalent to the difference of a general complete intersection and $C'$.

Let $\phi: Y \to X$ be the Fano fibration from Proposition 4.3. Since the generic fiber of $\phi$ is rationally connected, the base change $\phi_C:Y_C \to C$ of $\phi$ to $C$ has a section $s$ by \cite{GHS}. Since $Y \to M_{1,p}$ is a family of broken conics, $s(C)$ is rationally equivalent to a linear combination: $$\sum_i n_i\ell_i + \sum_j m_jZ_j,$$ where the $\ell_i$'s are lines, the $Z_j$'s curves contained in $\sigma_1(M_{1,p})$, and the $n_i$'s and $m_j$'s are integers. Therefore, in $\CH_1(X)$, the class of $C$ is equal to: $$[C]=\phi_*(\sum_i n_i\ell_i + \sum_j m_jZ_j)= \sum_i n_i \phi_*(\ell_i).$$
\end{proof}

We have two main applications of Theorem 4.4. The first one, is to $2$-Fano weighted complete intersections.

\begin{cor}
Let $X_{d_1,...,d_c} \subset \PP(1,1,1,a_3,...,a_n)$ be a smooth weighted complete intersection of dimension at least $3$ which is not isomorphic to a (standard) projective space. Assume that $$d_1+...+d_c \leq a_3+...+a_n$$ and that $$d_1^2+...+d_c^2 \leq 3+a_3^2+...+a_n^2.$$ Then $CH_1(X_{d_1,...,d_c})$ is generated by lines. If furthermore the Fano variety of lines is irreducible, $CH_1(X_{d_1,...,d_c}) \simeq \mathbb{Z}$.
\end{cor}
\begin{proof}
The result follows from the fact that such weighted complete intersections are broken-conic-connected by \cite{Min}, proof of Proposition 4.1, and that $\Delta$ is connected since the support of $\Delta$ is an effective ample divisor by \cite[Lemma 4.2]{Min}. Finally, if the Fano variety of lines is irreducible, it is rationally connected. In fact, the general fiber $F$ of $ev_1^1: \MM_{0,1}(X,\alpha) \to X$ is smooth and irreducible by \cite[Proposition 3.2]{Min}, and since $X$ is $2$-Fano, $c_1(F)$ is positive by \cite[Proposition 1.3]{AC}. 
Therefore $\MM_{0,1}(X,\alpha)$ is rationally connected by \cite{GHS}, and hence $\MM_{0,0}(X,\alpha)$ is rationally connected.
\end{proof}

An explicit example is given by cyclic covers of projective space:

\begin{cor}
Let $X$ be a smooth, degree $r$ cover of $\PP^{n-1}$ branched along a hypersurface of degree $ar$ (with $a,r \geq 2$). If $n \geq a^2(r^2-1)$, $CH_1(X)$ is generated by lines. If furthermore $X$ is general, $\CH_1(X) \simeq \mathbb{Z}$.
\end{cor}
\begin{proof}
A cyclic cover $X$ as above can be realized as a smooth weighted hypersurface $X_{ar} \subset \PP(1,...1,a)$. The result then follows by Corollary 4.5, and by the fact that the Fano variety of lines is irreducible for $X$ general by \cite[Theorem 1.1.7]{Smi}.
\end{proof}

The second application is to $2$-Fano, conic-connected varieties of high enough index.

\

By the classification of Ionescu and Russo (\cite[Theorem 2.2]{IR}), if $X \subset \PP^n$ is smooth, conic-connected, linearly normal and non-degenerate, then $X$ is Fano with $Pic(X) \simeq \mathbb{Z}<\OO_X(1)>$ and index $\iota_X \geq \frac{dim(X)+1}{2}$, with precisely $4$ exceptions -- let us call them \em exceptional examples \em -- that we can consider separately:

(i) $X \simeq \PP^n$, hence $\CH_1(X) \simeq \mathbb{Z}$;

(iii) $X$ is isomorphic to a projective bundle over projective space, hence $\CH_1(X) \simeq \mathbb{Z} \times \mathbb{Z}$;

(iii) $X$ is isomorphic to a hyperplane section of a Segre embedding of a product $\PP^a \times \PP^b$, hence (by the Lefschetz Hyperplane Theorem) $\CH_1(X) \simeq \mathbb{Z} \times \mathbb{Z}$;

(iv) $X$ is isomorphic to a product $\PP^a \times \PP^b$, hence $\CH_1(X) \simeq \mathbb{Z} \times \mathbb{Z}$.

\

It is therefore natural to consider conic-connected varieties of high index. In fact, if the index $\iota_X$ is high enough -- either with respect to $dim(X)$ or to $n$ -- we have the following:

\begin{cor}
Let $X \subset \PP^n$ be a smooth, $2$-Fano, non-linear, conic-connected variety. Let $\OO_X(1)$ be the line bundle induced by hyperplane sections. Assume that either: 

1) $\iota_X \geq \frac{n+4}{2}$, or

2) $\iota_X \geq \frac{3dim(X)+5}{4}$.

Then $\CH_1(X)$ is generated by lines. If furthermore the Fano variety of lines is irreducible, $\CH_1(X) \simeq \mathbb{Z}$.
\end{cor}
\begin{proof}
Since the index is at least $\frac{dim(X)+3}{2}$ in both cases considered in the hypothesis, $X$ is \em broken\em -conic-connected: namely, the space of conics through $2$ general points has dimension $-K_X \cdot 2\alpha - \dim(X) -1 \geq 2$, and therefore such a conic will break. 

Again, since the index is at least $\frac{dim(X)+3}{2}$, by \cite[Proposition 2.5]{IR} the general fiber $F$ of $ev_1^1: \MM_{0,1}(X,\alpha) \to X$ is smooth and irreducible. Furthermore, if the Fano variety of lines is irreducible, it is rationally connected; namely, since $X$ is $2$-Fano, $c_1(F)$ is positive by \cite[Proposition 1.3]{AC}. Therefore $\MM_{0,1}(X,\alpha)$ is rationally connected by \cite{GHS}, and hence $\MM_{0,0}(X,\alpha)$ is rationally connected.

\

Thus, if we can prove that $\Delta$ is connected, the result follows from Theorem 4.4. Connectedness is a consequence of the hypothesis on the index.

\

\textbf{Case 1:} Assume that $\iota_X \geq \frac{n+4}{2}$. We can use the Fulton-Hansen Connectedness Theorem to prove that $\Delta$ is connected, as follows. Note first that, since the index is at least $\frac{dim(X)+3}{2}$, by \cite[Proposition 2.5]{IR} the general fiber of $ev_1^1: \MM_{0,1}(X,\alpha) \to X$ is smooth and irreducible.

\

Consider the evaluation map $e_2: \MM_{0,2}(X,1) \to X \times X$, and denote the diagonal in $X \times X$ as $\Delta_X$.
Define $M_{p_1,\bullet} :=e_2^{-1}(\{p_1\} \times X)$, $M_{\bullet, p_2} :={e_2}^{-1}(X \times \{p_2\})$. Then $M_{p_1,\bullet}$ and $M_{\bullet, p_2}$ are smooth and connected, since the general fiber of $ev_1: \MM_{0,1}(X,\alpha) \to X$ is smooth and connected.

\

We have maps $M_{p_1,\bullet} \to X$ and $M_{\bullet, p_2} \to X$, given by evaluation at the unspecified marked point point $\bullet$. Their product defines a map $$e_{1,2}: M_{p_1,\bullet} \times M_{\bullet, p_2} \to X \times X,$$ which is generically finite, and with the property that $e_{1,2}^{-1}(\Delta_X)=\Delta$. 

\

Now consider the Stein factorization:

\[
\begin{tikzcd}
 & M_{p_1,\bullet,\bullet,p_2}  \arrow{dr}{e} \\
M_{p_1,\bullet} \times M_{\bullet, p_2} \arrow{ur}{\eta} \arrow{rr}{e_{1,2}} && X \times X
\end{tikzcd}
\]

Since $M_{p_1,\bullet,\bullet,p_2}$ is irreducible of dimension $dim(M_{p_1,\bullet}) + dim(M_{\bullet, p_2})=2(-K_X \cdot \alpha -2)$, and $\iota_X \geq \frac{n+4}{2}$, we have that $dim(e(M_{p_1,\bullet,\bullet,p_2}))> n$. Therefore the Fulton-Hansen Theorem implies that $e^{-1}(\Delta_{\mathbb{P}^n})=e^{-1}(\Delta_X)$ is connected. Thus (since $\eta$ has connected fibers) $e_{1,2}^{-1}(\Delta_X)=\Delta$ is connected.

\

\textbf{Case 2:} Assume that $\iota_X \geq \frac{3dim(X)+5}{4}$. Let $(p,q) \in X \times X$ be a general point. Then the line $L$ through $p$ and $q$ is not contained in $X$. Let $N$ denote the normal bundle $N_{L/\PP^n}$, and $\mathcal{I}_p$, $\mathcal{I}_q$ the ideal sheaves of $p$ and $q$ in $L$. We can consider a subspace $S_p$ of $H^0( L , N(-1))$ corresponding to sections whose image under the map $H^0(L, \mathcal{I}_q N) \to N |_p$ belongs to $T_pX$ (the tangent space of $X$ at $p$). Similarly, we can define $S_q$ by switching the roles of $p$ and $q$. Then we have a subspace $S_{p,q} \subset H^0(L, N(-1))$, of dimension at most $dim(X)$, that intuitively parametrizes infinitesimal deformations of $L$ along directions that at $p$ and $q$ belong to the tangent space of $X$ at $p$ and $q$ respectively.

\

Let $\F$ be the fiber of the evaluation map $ev_2^2: \MM_{0,2}(X, 2\alpha) \to X \times X$ over $(p,q)$. We then have a tangency morphism $\tau: \F \to \mathbb{P}(S_{p,q})$ that maps a conic $C$ to $Span(C)/\overline{pq}$, where $Span(C)$ is the $2$-plane generated by the image of the conic in $\PP^n$, and $\overline{pq}$ is the line through $p$ and $q$. Since $L$ is not contained in $X$, $\tau$ is finite. Now assume that $\Delta$ has $2$ irreducible components $\Delta_1$ and $\Delta_2$. Since $\Delta$ has pure codimension $1$, $$dim(\Delta)=dim(\Delta_1)= dim(\Delta_2).$$ Since $\iota_X \geq \frac{3dim(X)+5}{4}$, $$dim(\Delta)=-K_X\cdot 2\alpha -dim(X)-2 \geq \frac{dim(X)+1}{2}.$$ Therefore, since $ \mathbb{P}(S_{p,q})\leq dim(X)-1$ and $\tau$ is finite, there is a $2$-dimensional family of $2$-planes through $L$ whose intersection with $X$ contains conics parametrized by both $\Delta_1$ and $\Delta_2$. This family is parametrized by a curve $B$, that we can assume to be smooth up to normalizing it.

\

Now let us focus our attention on the point $p$. For every $b \in B$, there is a $\PP^1$ of lines through $p$ contained in the $2$-plane corresponding to $b$. As $b$ varies, we get a $\PP^1$-bundle $M$ over $B$, with $3$ special cross-sections: $\rho_L$ corresponding to $L$, $\rho_1$ corresponding to the lines through $p$ of broken conics parametrized by $\Delta_1$, and $\rho_2$ corresponding to the lines through $p$ of broken conics parametrized by $\Delta_2$. Note that the latter two cross-sections are disjoint from the one corresponding to $L$. Now assume that $\rho_1(B)$ and $\rho_2(B)$ are disjoint too, which implies that $M \simeq B \times \PP^1$ is the trivial $\PP^1$-bundle. Consider the natural morphism $\mu: M \to \PP(T_p(\PP^n))$ that maps a point $[\ell] \in M$ parametrizing a line $\ell$ through $p$ to the corresponding direction in the tangent space to $\PP^n$ at $p$:

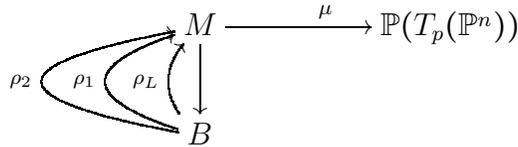
\[
\xymatrix{M \ar[rr]^{\mu} \ar[d] && \PP(T_p(\PP^n))\\
B \ar@/^1pc/[u]^{\rho_L} \ar@/^3pc/[u]^{\rho_1} \ar@/^5pc/[u]^{\rho_2}}
\]

Then $\mu(M)$ is a surface. Namely, since $\mu(\rho_L(B))=[L]$, $\mu(\rho_1(B))$ and $\mu(\rho_2(B))$ are all disjoint by construction, the only other possibility is for $\mu(M)$ to be a curve and for $\mu(\rho_L(B))$, $\mu(\rho_1(B))$ and $\mu(\rho_2(B))$ to be $3$ distinct points. But this means that the points in $B$ parametrize lines on a fixed plane, which implies that some conic in $\tau(\Delta)$ contains $L$, a contradiction ($L$ is not contained in $X$). Therefore $\mu(M)$ is a surface and $\rho_L(B)$ is a curve contracted by $\mu$, which implies that $(\rho_L(B))^2 < 0$. But this is impossible, since $M$ is the trivial bundle and therefore $(\rho_L(B))^2=0$. In conclusion, $\rho_1(B)$ and $\rho_2(B)$ have to intersect.

\

Thus there are broken conics parametrized by $\Delta_1$ whose component through $p$ coincides with the component through $p$ of broken conics parametrized by $\Delta_2$. Since $dim(\Delta) \geq \frac{dim(X)+1}{2}$, there is furthermore a $1$-dimensional family of $2$-planes through $L$ whose intersection with $X$ contains conics parametrized by both $\Delta_1$ and $\Delta_2$ with the same component through $p$. 

By iterating the process above for this family, except now for the $\PP^1$ bundle given by lines through $q$, we find that there must be a broken conic parametrized by both $\Delta_1$ and $\Delta_2$; i.e., $\Delta_1 \cap \Delta_2 \neq \emptyset$. Since $\Delta$ is smooth, this is a contradiction. Therefore $\Delta$ is connected.

\end{proof}

Furthermore, note that if $X \subset \PP^n$ has Picard number $1$ and index $\iota_X > \frac{2}{3} dim(X)$, then $X$ is conic-connected by \cite[Proposition 2.5(iii)]{IR}. Therefore we have the following:

\begin{cor}
Let $X \subset \PP^n$ be a smooth, $2$-Fano variety with Picard number $1$ and index $$\iota_X \geq \frac{3dim(X)+5}{4}.$$ Then $\CH_1(X)$ is generated by lines. If furthermore the Fano variety of lines is irreducible, $\CH_1(X) \simeq \mathbb{Z}$.
\end{cor}

\bibliographystyle{amsalpha}
\bibliography{cycles}

\end{document}